\documentclass[11pt,a4paper]{amsart}

\usepackage{fancyhdr,color,graphicx,amsmath,amsfonts,amssymb,latexsym,bm,indentfirst,cite,amsthm,enumerate}
\usepackage[colorlinks=true,backref=page]{hyperref}

\newtheorem{Theorem}{Theorem}[section]
\newtheorem{Definition}[Theorem]{Definition}
\newtheorem{Lemma}[Theorem]{Lemma}

\newtheorem{Proposition}[Theorem]{Proposition}

\numberwithin{equation}{section}
\newtheorem{theoremalph}{Theorem}

\begin{document}
   \newpage
   \title[Uniform Diophantine approximation]{Uniform Diophantine approximation related to beta-transformations}
 
   \author{Wanlou Wu}
   \address[Wanlou Wu]{School of Mathematics and Statistics, Jiangsu Normal University, Xuzhou, jiangsu, 221116, PR China}
   \email{wuwanlou@163.com}
  
  

\begin{abstract}
   For any $\beta>1$, let $T_\beta$ be the classical $\beta$-transformations. Fix $x_0\in[0,1]$ and a nonnegative real number $\hat{v}$, we compute the Hausdorff dimension of the set of real numbers $x\in[0,1]$ with the property that, for every sufficiently large integer $N$, there is an integer $n$ with $1\leq n\leq N$ such that the distance between $T_\beta^nx$ and $x_0$ is at most equal to $\beta^{-N\hat{v}}$. This work extends the result of Bugeaud and Liao \cite{YLiao2016} to every point  $x_0$ in unit interval. 
\end{abstract}

\maketitle  	
	
\section{Introduction and results}
   Diophantine approximation has been widely studied by mathematicians. In 1842, Dirichlet \cite{D1842} proved an illustrious theorem as follows.\\
   {\bf Dirichlet Theorem} For any two real numbers $\theta,~Q$ with $Q\geq 1$, there is an integer $n\in[1, Q]$ such that $$\lVert n \theta\rVert<Q^{-1},$$ where $\lVert\xi\rVert$ denotes the distance from $\xi$ to the nearest integer. 
  
   Dirichlet Theorem is called a \emph{uniform approximation theorem} in \cite[p.p. 2]{W12}. A weak form of Dirichlet Theorem, called an \emph{asymptotic approximation theorem} in \cite[p.p. 2]{W12}, which was often refered to as a corollary of Dirichlet Theorem in the litterature has already existed in the book of Legendre \cite[1808, p.p. 18-19]{L2009}: for any real number $\theta$, there are infinitely many $n\in\mathbb{N}$ such that $$\lVert n\theta\rVert<n^{-1}.$$ For the general case, Khintchine in 1924 \cite{K1924} showed that for a positive function $\psi:\mathbb{N}\rightarrow\mathbb{R}^+$, if $x\mapsto x\psi(x)$ is non-increasing, then the set $$\mathcal{L}_\psi:=\left\{\theta\in\mathbb{R}:\lVert n\theta\rVert<\psi(n),\text{ for infinitely many }n\in\mathbb{N}\right\}$$ has Lebesgue measure zero if the series $\sum\psi(n)$ converges and has full Lebesgue measure otherwise. In the case where the set has Lebesgue measure zero, it is natural to calculate the Hausdorff dimension of $\mathcal{L}_\psi$. The first result on the Hausdorff dimension of $\mathcal{L}_\psi$ dates back to Jarn\'{\i}k-Bosicovitch Theorem \cite{B34,J29}. It was shown that for any $\tau>1$, one has $${\rm dim}_H\left(\left\{\theta\in\mathbb{R}:\lVert n\theta\rVert<\dfrac{1}{n^\tau},\text{ for infinitely many }n\in\mathbb{N}\right\}\right)=\dfrac{2}{1+\tau},$$ where ${\rm dim}_H(\cdot)$ denotes the Hausdorff deminsion of a set.
   
   In analogy with the classical Diophantine approximation, Hill and Velani \cite{HV1995} studied the approximation properties of the orbits of a dynamical system and introduced the so called \emph{shrinking target problems}: for a measure preserving dynamical system $(M,\mu,T)$ with a metric $d$ and a positive function $\psi$, define the set of all \emph{$\psi$-well approximable} points by $x_0$ as $$\mathcal{L}(T,\psi,x_0):=\{x\in M:d(T^nx,x_0)<\psi(n),\text{ for infinitely many $n\in\mathbb{N}$}\},$$ what is the size (Lebesgue measure, Hausdorff dimension) of $\mathcal{L}(T,\psi,x_0)$? They studied the case where $T$ is an expanding rational map of the Riemann sphere $\overline{\mathbb{C}}=\mathbb{C}\cup\{\infty\}$. 
   
   In this papper, we are interested in the approximation properties of the orbits of $\beta$-transformations. The $\beta$-transformation $T_\beta~(\beta>1)$ on $[0,1)$ is defined by $$T_\beta(x):=\beta x-\lfloor\beta x\rfloor,$$ where $\lfloor\cdot\rfloor$ is the integer part function. For any positive function $\psi:\mathbb{N}\rightarrow\mathbb{R}^+$, define the set of \emph{$\psi$-well asymptotically approximable} points by $x_0$ as $$\mathcal{L}(\psi,x_0):=\left\{x\in[0,1]:\lvert T_\beta^n x-x_0\rvert<\psi(n),\text{ for infinitely many $n\in\mathbb{N}$}\right\}.$$ By \cite[Theorem 2A, B, C]{P67}, the set $\mathcal{L}(\psi,x_0)$ has Lebesgue measure zero if and only if the series $\sum\psi(n)$ converges. Shen and Wang \cite[Theorem 1.1]{SW2013} showed that for any real number $\beta>1$ and any point $x_0\in[0,1]$, one has $${\rm dim}_H\left(\mathcal{L}(\psi, x_0)\right)=\dfrac{1}{1+v},\quad\text{where $v:=\liminf\limits_{n\rightarrow\infty}\dfrac{-\log_\beta\psi(n)}{n}$}.$$ 
   
   Parallel to the asymptotic approximation theorem, it is also worth of studying the uniform approximation properties as in Dirichlet Theorem. The uniform Diophantine approximation related to $\beta$-transformations was studied by Bugeaud and Liao \cite{YLiao2016}. For $x\in[0,1)$, let $$v_\beta(x):=\sup\left\{v\geq 0:T^n_\beta x<(\beta^n)^{-v},\text{ for infinitely many $n\in\mathbb{N}$}\right\},$$ $$\hat{v}_\beta(x):=\sup\left\{v\geq 0:\forall~N\gg 1,~T^n_\beta x<(\beta^N)^{-v}\text{ has a solution $n\in[0,N]$}\right\}.$$ The exponents $v_\beta$ and $\hat{v}_\beta$ were introduced in \cite{AB10}(see also \cite[Ch.7]{B2012}). Bugeaud and Liao \cite{YLiao2016} proved the following theorem.\\{\bf Theorem BL} (\cite[Theorem 1.4]{YLiao2016}) \emph{For any $v\in(0,\infty)$ and any $\hat{v}\in(0,1)$, if $v<\hat{v}/(1-\hat{v})$, then the set $$\left\{x\in[0,1]:v_\beta(x)= v\right\}\cap\left\{x\in[0,1]:\hat{v}_\beta(x)\geq\hat{v}\right\}$$ is empty. Otherwise, $${\rm dim}_H\left(\{x\in[0,1]:v_\beta(x)= v\}\cap\{x\in[0,1]:\hat{v}_\beta(x)=\hat{v}\}\right)=\dfrac{v-\hat{v}-v\hat{v}}{(1+v)(v-\hat{v})}.$$}{\bf Theorem BL} can be considered as the special case where $x_0=0$. The aim of this paper is to study the Diophantine approximation sets in \cite{YLiao2016} for any fixed $x_0\in(0,1]$.
\begin{Definition}\label{defexp}
   Let $\beta>1$, fix $x_0\in[0,1]$. For any $x\in[0,1]$, denote by $\mathcal{V}_\beta(x,x_0)$ the supremum of the real numbers $v$ for which the equation $$\lvert T_\beta^n x-x_0\rvert<(\beta^n)^{-v}$$ has infinitely many solutions in integers $n\in\mathbb{N}$. Denote by $\hat{\mathcal{V}}_\beta(x,x_0)$ the supremum of the real numbers $\hat{v}$ for which, for every sufficiently large integer $N$, the equation $$\lvert T_\beta^n x-x_0\rvert<(\beta^N)^{-\hat{v}}$$ has a solution $n\in\mathbb{N}$ with $1\leq n\leq N$.
\end{Definition}      
   
   Our main results are the following Theorems \ref{A} and \ref{B}.
\begin{theoremalph}\label{A}
	Let $\beta>1$. For any $x_0\in[0,1]$, any $v\in(0,\infty)$ and any $\hat{v}\in(0,1)$, if $v<\hat{v}/(1-\hat{v})$, then the set $$\left\{x\in[0,1]:\mathcal{V}_\beta(x,x_0)=v\right\}\cap\left\{x\in[0,1]:\hat{\mathcal{V}}_\beta(x,x_0)\geq\hat{v}\right\}$$ is empty. Otherwise, the set $$\left\{x\in[0,1]:\mathcal{V}_\beta(x,x_0)= v\right\}\cap\left\{x\in[0,1]:\hat{\mathcal{V}}_\beta(x,x_0)=\hat{v}\right\}$$ has the Hausdorff dimension  $$\dfrac{v-\hat{v}-v\hat{v}}{(1+v)(v-\hat{v})}.$$
\end{theoremalph}
     
\begin{theoremalph}\label{B}
   Let $\beta>1$, the set $$\left\{x\in[0,1]:\hat{\mathcal{V}}_\beta(x,x_0)=0\right\}$$ is of full Lebesgue measure. If $0<\hat{v}\leq1$, then $${\rm dim}_H\left( \left\{x\in[0,1]:\hat{\mathcal{V}}_\beta(x,x_0)=\hat{v}\right\}\right)=\left(\dfrac{1-\hat{v}}{1+\hat{v}}\right)^2.$$ Otherwise, the set $$\left\{x\in[0,1]:\hat{\mathcal{V}}_\beta(x,x_0)>1\right\}$$ is countable. 
\end{theoremalph}

   Persson and Schmeling \cite{PS08} gave another point of view, by letting $\beta$ vary and considering the $\beta$-expansions of $1$. They \cite[Theorem 14]{PS08} showed that for any $1<\beta_0<\beta_1<2$ and any $v\geq0$, one has $${\rm dim}_H\left( \left\{\beta\in(\beta_0,\beta_1):\mathcal{V}_\beta(1,x_0)=v\right\}\right)=\dfrac{1}{1+v}.$$ In \cite{LPWW}, the assumption $\beta_1<2$ is removed. In the same way as {\bf Theorem BL}, Bugeaud and Liao \cite{YLiao2016} also proved that for any $v\in(0,\infty)$ and any $\hat{v}\in(0,1)$, if $v<\hat{v}/(1-\hat{v})$, then $$\left\{\beta>1:v_\beta(1)=v\right\}\cap\left\{\beta>1:\hat{v}_\beta(1)\geq\hat{v}\right\}=\emptyset,$$ Otherwise, $${\rm dim}_H\left(\{\beta>1:v_\beta(1)= v\}\cap\left\{\beta>1:\hat{v}_\beta(1)=\hat{v}\right\}\right)=\dfrac{v-\hat{v}-v\hat{v}}{(1+v)(v-\hat{v})}.$$ When considering the exponents $\mathcal{V}_\beta(1,x_0)$ and $\hat{\mathcal{V}}_\beta(1,x_0)$, we obtain the following Theorems \ref{C} and \ref{D}.  
\begin{theoremalph}\label{C}
   For any $x_0\in[0,1]$, any $v\in(0,\infty)$ and any $\hat{v}\in(0,1)$, if $v<\hat{v}/(1-\hat{v})$, then the set $$\left\{\beta>1:\mathcal{V}_\beta(1,x_0)=v\right\}\cap\left\{\beta>1:\hat{\mathcal{V}}_\beta(1,x_0)\geq\hat{v}\right\}$$ is empty. Otherwise, $${\rm dim}_H\left(\{\beta>1:\mathcal{V}_\beta(1,x_0)= v\}\cap\left\{\beta>1:\hat{\mathcal{V}}_\beta(1,x_0)=\hat{v}\right\}\right)=\dfrac{v-\hat{v}-v\hat{v}}{(1+v)(v-\hat{v})}.$$    
\end{theoremalph}
  
\begin{theoremalph}\label{D}
   For any $x_0\in[0,1]$ and any $\hat{v}\in[0,1]$, one has $${\rm dim}_H\left( \left\{\beta>1:\hat{\mathcal{V}}_\beta(1,x_0)\geq\hat{v}\right\}\right)={\rm dim}_H\left( \left\{\beta>1:\hat{\mathcal{V}}_\beta(1,x_0)=\hat{v}\right\}\right)$$ and $${\rm dim}_H \left(\{\beta>1:\hat{\mathcal{V}}_\beta(1,x_0)=\hat{v}\}\right)=\left(\dfrac{1-\hat{v}}{1+\hat{v}}\right)^2.$$ 	
\end{theoremalph}
   
   Our paper is organized as follows. We recall some classical results of the theory of $\beta$-transformations in Section $2$. Theorems \ref{A} and \ref{B} are proved in Section $3$. Section $4$ establishes Theorems \ref{C} and \ref{D}.

\medskip  
\section{Beta-transformations}
  
   Beta-expansion was introduced by R\'{e}nyi \cite{Ren57} in 1957. For any $\beta>1$, the $\beta$-transformation $T_{\beta}$ on $[0,1)$ is defined by $$T_\beta x=\beta x-\lfloor\beta x\rfloor,$$ where $\lfloor\xi\rfloor$ denotes the largest integer less than or equal to $\xi$. Let
\begin{equation*} \lceil\beta\rfloor=
  \begin{cases}
  \beta -1, &  \text{if}~\beta \text{ is a positive integer},\\
  \lfloor \beta \rfloor, &  \text{otherwise}.
  \end{cases}
\end{equation*}
  
\begin{Definition}
   A sequence $\{\varepsilon_n:\varepsilon_n=\varepsilon_n(x,\beta)\}_{n\geq 1}\in\mathcal{A}^{\mathbb{N}}:=\{0,1,\cdots,\lceil\beta\rfloor\}^{\mathbb{N}}$ is called the $\beta$-expansion of a number $x\in[0,1)$, if
\begin{equation}\label{E1}
   x=\dfrac{\varepsilon_1}{\beta}+\dfrac{\varepsilon_2}{\beta^2}+\cdots+\dfrac{\varepsilon_n}{\beta^n}+\cdots,
\end{equation}
   where $\varepsilon_n(x,\beta)=\lfloor\beta T^{n-1}_\beta x\rfloor$, for all positive integers $n\in\mathbb{N}$. We also write $$d_\beta(x)=\left(\varepsilon_1,\cdots,\varepsilon_n,\cdots\right).$$
\end{Definition}
  
   We can extend the definition of the $\beta$-transformation to the point $1$ as:$$T_\beta 1=\beta-\lfloor\beta \rfloor.$$ One can obtain $$ 1=\dfrac{\varepsilon_1(1,\beta)}{\beta}+\dfrac{\varepsilon_2(1,\beta)}{\beta^2}+\cdots+\dfrac{\varepsilon_n(1,\beta)}{\beta^n}+\cdots,$$ where $\varepsilon_n(1,\beta)=\lfloor\beta T^{n-1}_\beta 1\rfloor,~\text{for all positive integers $n\in\mathbb{N}$}$. We also write $$ d_\beta(1)=\left(\varepsilon_1(1,\beta),\cdots,\varepsilon_n(1,\beta),\cdots\right).$$ If $d_\beta(1)$ is finite, i.e. there is an integer $m>0$ such that $\varepsilon_m(1,\beta)\neq 0$ and $\varepsilon_i(1,\beta)=0$ for all $i>m$, then $\beta$ is called a \emph{simple Parry number}. In this case, the infinite $\beta$-expansion of $1$ is defined as $$(\varepsilon^\ast_1(\beta),\varepsilon^\ast_2(\beta),\cdots,\varepsilon^\ast_n(\beta),\cdots):=(\varepsilon_1(1,\beta),\varepsilon_2(1,\beta),\cdots,\varepsilon_m(1,\beta)-1)^\infty,$$ where $(\omega)^\infty$ denotes the periodic sequence. If $d_\beta(1)$ is infinite, then we define $$(\varepsilon^\ast_1(\beta),\varepsilon^\ast_2(\beta),\cdots,\varepsilon^\ast_n(\beta),\cdots):=(\varepsilon_1(1,\beta),\varepsilon_2(1,\beta),\cdots,\varepsilon_n(1,\beta),\cdots).$$
   
   Endow the set $\mathcal{A}^{\mathbb{N}}$ with the product topology and define the one-sided shift operator $\sigma$ as $$\sigma\left((\omega_n)_{n\geq1}\right):=(\omega_{n+1})_{n\geq1},$$ for any infinite sequence $(\omega_n)_{n\geq1}\in\mathcal{A}^{\mathbb{N}}$. The lexicographical order $<_{lex}$ on $\mathcal{A}^{\mathbb{N}}$ is defined as $$\omega=(\omega_1,\omega_2,\cdots)<_{lex}\omega'=(\omega'_1,\omega'_2,\cdots),$$ if $\omega_1<\omega'_1$ or if there is an integer $k\geq 2$ such that for all $1\leq i< k$, $\omega_i=\omega'_i$ but $\omega_k<\omega'_k$. Denote by $\omega\leq_{lex}\omega'$ if $\omega<_{lex}\omega'$ or $\omega=\omega'$.    
   
   A finite word $(\omega_1,\omega_2,\cdots,\omega_n)$ is called $\beta$-admissible, if there is $x\in[0,1]$ such that the $\beta$-expansion of $x$ begins with $(\omega_1,\omega_2,\cdots,\omega_n).$ An infinite sequence $(\omega_1,\omega_2,\cdots,\omega_n,\cdots)$ is called $\beta$-admissible, if there is $x\in[0,1]$ such that the $\beta$-expansion of $x$ is $(\omega_1,\omega_2,\cdots,\omega_n,\cdots)$. An infinite sequence $(\omega_1,\omega_2,\cdots,\omega_n,\cdots)$ is self-admissible, if $$\sigma^k(\omega_1,\omega_2,\cdots,\omega_n,\cdots)\leq_{lex}(\omega_1,\omega_2,\cdots,\omega_n,\cdots),~\text{for }k\geq0.$$ Denote by $\Sigma_\beta$ the set of all infinite $\beta$-admissible sequences and denote by $\Sigma^n_\beta$ the set of all $\beta$-admissible sequences with length $n$. The $\beta$-admissible sequences are characterized by Parry \cite{P1960} and R\'{e}nyi \cite{Ren57}.
\begin{Theorem}\label{Admissible}
   Let $\beta>1$,
\begin{enumerate}[(1)]
  \item (\cite[Lemma 1]{P1960})
   A word $\omega=(\omega_n)_{n\geq 1}\in\Sigma_\beta$ if and only if $$\sigma^k(\omega)\leq_{lex}(\varepsilon^\ast_1(\beta),\varepsilon^\ast_2(\beta),\cdots,\varepsilon^\ast_n(\beta),\cdots),~\text{for all $k\geq 0$}.$$ 
  	
  \item (\cite[Lemma 3]{P1960}) For any $x_1,~x_2\in[0,1]$, $x_1<x_2$ if and only if $$d_\beta(x_1)<_{lex}d_\beta(x_2).$$
  	
  \item (\cite[Lemma 4]{P1960}) For any $\beta_2>\beta_1>1$, one has $$\Sigma^n_{\beta_1}\subseteq\Sigma^n_{\beta_2},\quad\Sigma_{\beta_1}\subseteq\Sigma_{\beta_2}.$$
\end{enumerate}	
\end{Theorem} 

\begin{Theorem}\label{cardinality}(\cite[Theorem 2]{Ren57})
   For any $\beta>1$, one has $$\beta^n\leq\sharp \Sigma^n_\beta\leq\dfrac{\beta^{n+1}}{\beta-1},$$ where $\sharp$ denotes the cardinality of a finite set.
\end{Theorem} 

   For every $(\omega_1,\cdots,\omega_n)\in\Sigma^n_\beta$, we call $$I_n(\omega_1,\cdots,\omega_n):=\{x\in[0,1]:d_\beta(x) \text{ starts with }\omega_1,\cdots,\omega_n\}$$ an \emph{$n$-th order basic interval} with respect to $\beta$. Denote by $I_n(x)$ the $n$-th order basic interval containing $x$. The basic intervals are also called \emph{cylinders}. It is crucial to estimate the lengths of the basic intervals. We will use the key notion of \textquotedblleft full cylinder\textquotedblright introduced by Fan and Wang \cite{FW2012}. For any $(\omega_1,\cdots,\omega_n)\in\Sigma^n_\beta$, a basic interval $I_n(\omega_1,\cdots,\omega_n)$ is said to be full if its length is $\beta^{-n}$. Denote by $\lvert I_n(\omega_1,\cdots,\omega_n)\rvert$ the length of the $n$-th order basic interval.
\begin{Proposition}\label{full}
   (\cite[Lemma 3.1]{FW2012} and \cite[Lemma 2.5]{SW2013})
	 
   For any $(\omega_1,\cdots,\omega_n)\in\Sigma^n_\beta$, the following statements are equivalent:
\begin{enumerate}[(1)]
   \item $I_n(\omega_1,\cdots,\omega_n)$ is a full basic interval.
   \item $T^n_\beta I_n(\omega_1,\cdots,\omega_n)=[0,1)$.
   \item For any $\omega'=(\omega'_1,\cdots,\omega'_m)\in\Sigma^m_\beta$, the concatenation $$(\omega_1,\cdots,\omega_n,\omega'_1,\cdots,\omega'_m)\in\Sigma^{n+m}_\beta, \text{ i.e., is $\beta$-admissible.}$$ 
\end{enumerate}
\end{Proposition}

\begin{Proposition}\label{fullc}
	(\cite[Corollary 2.6]{SW2013})
\begin{enumerate}[(1)]
   \item If $(\omega_1,\cdots,\omega_{n+1})$ is a $\beta$-admissible sequence with $\omega_{n+1}\neq0$, then $$I_{n+1}(\omega_1,\cdots,\omega'_{n+1})$$ is full for any $0\leq \omega'_{n+1}<\omega_{n+1}$. 
   \item For every $\omega\in\Sigma^n_\beta$, if $I_n(\omega)$ is full, then for any $\omega'\in\Sigma^m_\beta$, one has $$\lvert I_{n+m}(\omega,\omega')\rvert=\lvert I_n(\omega)\rvert\cdot\lvert I_m(\omega')\rvert=\dfrac{\lvert I_m(\omega')\rvert}{\beta^n}.$$
   \item For any $\omega\in\Sigma^n_\beta$, if $I_{n+m}(\omega,\omega')$ is a full basic interval contained in $I_n(\omega)$ with the smallest order, then $$\lvert I_{n+m}(\omega,\omega')\rvert\geq\dfrac{\lvert I_n(\omega)\rvert}{\beta}.$$ 
\end{enumerate}
\end{Proposition}

   Next, we define a sequence of numbers $\beta_N$ approaching to $\beta$. Given the infinite $\beta$-expansion $(\varepsilon^\ast_1(\beta),\varepsilon^\ast_2(\beta),\cdots,\varepsilon^\ast_n(\beta),\cdots)$ of $1$. For any $\varepsilon^\ast_N(\beta)>0$, let $\beta_N$ be the unique real solution of the equation
\begin{equation}\label{ED1}
   1=\dfrac{\varepsilon^\ast_1(\beta)}{z}+\cdots+\dfrac{\varepsilon^\ast_N(\beta)}{z^N}.
\end{equation} 
   Therefore, $\beta_N<\beta$ and the sequence $\{\beta_N:N\geq 1\}$ increases and converges to $\beta$ when $N$ tends to infinity.

\begin{Lemma}\label{length}
   (\cite[Lemma 2.7]{SW2013}) For every $\omega\in\Sigma^n_{\beta_N}$ viewed as an element of $\Sigma^n_\beta$, one has $$\dfrac{1}{\beta^{n+N}}\leq\lvert I_n(\omega_1,\cdots,\omega_n)\rvert\leq\dfrac{1}{\beta^n}.$$
\end{Lemma}

\medskip
\section{Proofs of Theorems \ref{A} and \ref{B} }\label{sec2}

   For $\beta>1$ and $x_0\in[0,1]$, by the definitions of $\mathcal{V}_\beta(x,x_0)$ and $\hat{\mathcal{V}}_\beta(x,x_0)$, it can be checked that for every $x\in[0,1]$, we have $$\hat{\mathcal{V}}_\beta(x,x_0)\leq\mathcal{V}_\beta(x,x_0).$$ We first consider two special cases $\hat{\mathcal{V}}_\beta(x,x_0)=\mathcal{V}_\beta(x,x_0)=0$ and $\hat{\mathcal{V}}_\beta(x,x_0)=\mathcal{V}_\beta(x,x_0)=\infty$.  
\begin{Lemma}\label{spezero}
   If $\hat{\mathcal{V}}_\beta(x,x_0)=\mathcal{V}_\beta(x,x_0)=0$, then the set $$\left\{x\in[0,1]:\mathcal{V}_\beta(x,x_0)=0\right\}\cap\left\{x\in[0,1]:\hat{\mathcal{V}}_\beta(x,x_0)=0\right\}$$ is of full Lebesgue measure.  
\end{Lemma}

\begin{proof}
   Note that for any fixed $x_0\in[0,1]$ and any $x\in[0,1]$, we always have $$\hat{\mathcal{V}}_\beta(x,x_0)\leq\mathcal{V}_\beta(x,x_0).$$ Then, $$\left\{x\in[0,1]:\mathcal{V}_\beta(x,x_0)=0\right\}\subseteq\left\{x\in[0,1]:\hat{\mathcal{V}}_\beta(x,x_0)=0\right\}.$$ Hence, if for any fixed $x_0\in[0,1]$, we can prove that $\mathcal{V}_\beta(x,x_0)=0$, for Lebesgue almost every $x\in[0,1]$, then we prove the lemma.
   
   Now, we only need to prove $$m\left(\left\{x\in[0,1]:\mathcal{V}_\beta(x,x_0)>0\right\}\right)=0,$$ where $m(\cdot)$ denotes the Lebesgue measure of a set. In fact, we have $$\left\{x\in[0,1]:\mathcal{V}_\beta(x,x_0)>0\right\}=\cup^\infty_{k=1}\left\{x\in[0,1]:\mathcal{V}_\beta(x,x_0)>1/k\right\}$$ and $\left\{x\in[0,1]:\mathcal{V}_\beta(x,x_0)>1/k\right\}$ is a subset of $$\left\{x\in[0,1]:\lvert T^n_\beta x-x_0\rvert<\beta^{-n/k},\quad\text{for infinite many }n\in\mathbb{N}\right\}.$$ Since $\sum^\infty_{n=1}\beta^{-n/k}<\infty$ for any $k\geq1$, by \cite[Theorem 2A, B, C]{P67}, $$m\left(\left\{x\in[0,1]:\lvert T^n_\beta x-x_0\rvert<\beta^{-n/k},\quad\text{for infinite many }n\in\mathbb{N}\right\}\right)=0.$$ Therefore, for any $k\geq1$, we have $$m(\left\{x\in[0,1]:\mathcal{V}_\beta(x,x_0)>1/k\right\})=0.$$ Thus, $$m\left(\left\{x\in[0,1]:\mathcal{V}_\beta(x,x_0)>0\right\}\right)=0.$$  
\end{proof}

   When $\hat{\mathcal{V}}_\beta(x,x_0)=\mathcal{V}_\beta(x,x_0)=\infty$, we have the following Lemma \ref{infinite}.
\begin{Lemma}\label{infinite}
   If the number $\hat{\mathcal{V}}_\beta(x,x_0)=\infty$, then $$\left\{x\in[0,1]:\hat{\mathcal{V}}_\beta(x,x_0)=\infty\right\}=\bigcup^\infty_{n=1}\bigcup_{\omega\in\Sigma^n_\beta}\left\{x\in[0,1]:d_\beta(x)=(w,d_\beta(x_0))\right\}$$ is countable.  
\end{Lemma}	 

\begin{proof}
   For $\beta>1$ and $x_0\in[0,1]$, we suppose $$d_\beta(x_0)=(\epsilon_1,\epsilon_2,\cdots,\epsilon_n,\cdots).$$ If $x\in\cup^\infty_{n=1}\cup_{\omega\in\Sigma^n_\beta}\left\{x\in[0,1]:d_\beta(x)=(w,d_\beta(x_0))\right\}$, then there exists an integer $n_0$ such that $\lvert T_\beta^{n_0}x-x_0\rvert=0$. Therefore, for any $n\geq n_0$, there is $n_0\in[1,n]$ such that $$\lvert T_\beta^{n_0}x-x_0\rvert=0.$$ Thus, $\hat{\mathcal{V}}_\beta(x,x_0)=\infty$.
   
   Now, we prove $$\left\{x\in[0,1]:\hat{\mathcal{V}}_\beta(x,x_0)=\infty\right\}\subseteq\cup^\infty_{n=1}\cup_{\omega\in\Sigma^n_\beta}\left\{x\in[0,1]:d_\beta(x)=(w,d_\beta(x_0))\right\}.$$ By contrary, for any $x$ with $\hat{\mathcal{V}}_\beta(x,x_0)=\infty$, we suppose $$x\notin\cup^\infty_{n=1}\cup_{\omega\in\Sigma^n_\beta}\left\{x\in[0,1]:d_\beta(x)=(w,d_\beta(x_0))\right\}.$$ Denote the $\beta$-expansion of $x$ by $$x=\dfrac{a_1}{\beta}+\dfrac{a_2}{\beta^2}+\cdots+\dfrac{a_n}{\beta^n}+\cdots,$$ where $a_i\in\{0,\cdots,\lceil\beta\rfloor\}$, for all $i\geq1$. We can take two increasing sequences $\left\{n'_i:i\geq 1\right\}$ and $\left\{m'_i:i\geq1\right\}$ with the properties:
\begin{enumerate}[(1)]
   	\item For every $i\geq1$, one has $$a_{n'_i}>0,\quad a_{n'_i+1}=\epsilon_1,~a_{n'_i+2}=\epsilon_2,~\cdots,~a_{m'_i-1}=\epsilon_{m'_i-n'_i-1},\quad a_{m'_i}>0.$$
   	\item For every $a_n=0$, there is an integer $i$ such that $n'_i<n<m'_i$. 
\end{enumerate}
   By the choice of $\left\{n'_i: i\geq 1\right\}$ and $\left\{m'_i:i\geq1\right\}$, for every $i\geq1$, one has $n'_i<m'_i<n'_{i+1}$. Since $\hat{\mathcal{V}}_\beta(x,x_0)>0$, one has $$\limsup_{i\rightarrow\infty}(m'_i-n'_i)=\infty.$$ Taking $n_1=n'_1$ and $m_1=m'_1$, suppose $m_k,~n_k$ have been defined. Let $i_1=1$ and $i_{k+1}:=\min\{i>i_k:m'_i-n'_i> m_k-n_k\},~{\rm for }~k\geq1.$ Then, define $$n_{k+1}:=n'_{i_{k+1}},\quad m_{k+1}:=m'_{i_{k+1}}.$$ Therefore, the sequence $\{i_k:k\geq1\}$ is well defined. By this way, we obtain the subsequences $\{n_k:k\geq1\}$ and $\{m_k:k\geq1\}$ of $\left\{n'_i:i\geq 1\right\}$ and $\left\{m'_i:i\geq1\right\}$, respectively, such that the sequence $\{m_k-n_k: k\geq1\}$ is non-decreasing. Notice $\beta^{n_k-m_k}<\lvert T^{n_k}_\beta x-x_0\rvert<\beta^{n_k-m_k+1}$, we have $$\hat{\mathcal{V}}_\beta(x,x_0)=\liminf_{k\rightarrow\infty}\dfrac{m_k-n_k}{n_{k+1}}\leq1.$$ This contradicts our assumption $\hat{\mathcal{V}}_\beta(x,x_0)=\infty$. Thus, we have proved $$\left\{x\in[0,1]:\hat{\mathcal{V}}_\beta(x,x_0)=\infty\right\}\subseteq\cup^\infty_{n=1}\cup_{\omega\in\Sigma^n_\beta}\left\{x\in[0,1]:d_\beta(x)=(w,d_\beta(x_0))\right\}.$$ Therefore, $$\left\{x\in[0,1]:\hat{\mathcal{V}}_\beta(x,x_0)=\infty\right\}=\cup^\infty_{n=1}\cup_{\omega\in\Sigma^n_\beta}\left\{x\in[0,1]:d_\beta(x)=(w,d_\beta(x_0))\right\},$$ which implies that the set $\left\{x\in[0,1]:\hat{\mathcal{V}}_\beta(x,x_0)=\infty\right\}$ is countable.
\end{proof}   

\begin{Lemma}\label{empty}
	The set $\{x\in[0,1]:1<\hat{\mathcal{V}}_\beta(x,x_0)<\infty\}$ is empty.
\end{Lemma} 

\begin{proof}
   This follows from the proof of Lemma \ref{infinite}.
\end{proof}

   For $\beta>1$ and $x_0\in[0,1]$, we suppose $$d_\beta(x_0)=(\epsilon_1,\epsilon_2,\cdots,\epsilon_n,\cdots).$$ For the case $\mathcal{V}_\beta(x,x_0)\in(0,\infty)$ and $\hat{\mathcal{V}}_\beta(x,x_0)\in(0,1)$, we have the following discussion and complete the proof of Theorem \ref{A}.
   
\paragraph{{\bf Upper bound}} 
   For any $x\in[0,1]$, denote its $\beta$-expansion by  $$x=\dfrac{a_1}{\beta}+\dfrac{a_2}{\beta^2}+\cdots+\dfrac{a_n}{\beta^n}+\cdots.$$ Since $\hat{\mathcal{V}}_\beta(x,x_0)\in(0,1)$, by the same way as Lemma \ref{infinite}, we can take the maximal subsequences $\{n_k:k\geq1\}$ and $\{m_k:k\geq1\}$ of $\left\{n'_i:i\geq 1\right\}$ and $\left\{m'_i:i\geq1\right\}$, respectively. Similarly, notice that $$\beta^{n_k-m_k}<\lvert T^{n_k}_\beta x-x_0\rvert<\beta^{n_k-m_k+1}.$$ We have 
\begin{equation}\label{E2}
   \mathcal{V}_\beta(x,x_0)=\limsup_{n\rightarrow\infty}\dfrac{m_k-n_k}{n_k}=\limsup_{n\rightarrow\infty}\dfrac{m_k}{n_k}-1,
\end{equation} 

\begin{equation}\label{E3}
   \hat{\mathcal{V}}_\beta(x,x_0)\leq\liminf_{n\rightarrow\infty}\dfrac{m_k-n_k}{n_{k+1}}\leq\liminf_{n\rightarrow\infty}\dfrac{m_k-n_k}{m_k}=1-\limsup_{n\rightarrow\infty}\dfrac{n_k}{m_k}.
\end{equation}
   
   Since $\left(\limsup\dfrac{n_k}{m_k}\right)\cdot\left(\limsup\dfrac{m_k}{n_k}\right)\geq 1$, one has 
\begin{equation}\label{in1}
   \mathcal{V}_\beta(x,x_0)\geq\dfrac{\hat{\mathcal{V}}_\beta(x,x_0)}{1-\hat{\mathcal{V}}_\beta(x,x_0)},\qquad \hat{\mathcal{V}}_\beta(x,x_0)\leq\dfrac{\mathcal{V}_\beta(x,x_0)}{1+\mathcal{V}_\beta(x,x_0)}.
\end{equation}
   We can derive from \ref{in1} that if $v<\hat{v}/(1-\hat{v})$, then the set $$\left\{x\in[0,1]:\mathcal{V}_\beta(x,x_0)=v\right\}\cap\left\{x\in[0,1]:\hat{\mathcal{V}}_\beta(x,x_0)\geq\hat{v}\right\}$$ is empty. Otherwise, under the case where $\mathcal{V}_\beta(x,x_0)=v$ and $\hat{\mathcal{V}}_\beta(x,x_0)=\hat{v}$, take the two sequences $\{n_k:k\geq1\}$ and $\{m_k:k\geq1\}$ such that $$\mathcal{V}_\beta(x,x_0)=\lim_{n\rightarrow\infty}\dfrac{m_k-n_k}{n_k},\qquad \hat{\mathcal{V}}_\beta(x,x_0)\leq\liminf_{n\rightarrow\infty}\dfrac{m_k-n_k}{n_{k+1}}.$$ Given $0<\varepsilon<\hat{v}/2$, for $k$ large enough, one has
\begin{equation}\label{E4}
   (v-\varepsilon)n_k\leq m_k-n_k\leq(v+\varepsilon)n_k,
\end{equation}   
\begin{equation}\label{E5}
   m_k-n_k\geq(\hat{v}-\varepsilon)n_{k+1}.
\end{equation} 
   By inequality (\ref{E4}), one has $$(1+v-\varepsilon)m_{k-1}\leq(1+v-\varepsilon)n_k\leq m_k.$$ Therefore, the sequence $\{m_k:k\geq1\}$ increases at least exponentially. Since $n_k\geq m_{k-1}$ for every $k\geq2$, the sequence $\{n_k:k\geq1\}$ also increases at least exponentially. Thus, there is a positive constant $C$ such that $k\leq C\log_\beta n_k$. Combining (\ref{E4}) and (\ref{E5}), one obtains $$(\hat{v}-\varepsilon)n_{k+1}\leq(v+\varepsilon)n_k.$$ Thus, for $k$ large enough, there is an integer $n_0$ and a postive real number $\varepsilon_1$ small enough such that the sum of all lengths of the blocks of $0$ in the prefix of length $n_k$ of the infinite sequence $a_1a_2\cdots$ is at least equal to  
\begin{eqnarray*}
   n_k(\hat{v}-\varepsilon)\left\lgroup1+\dfrac{\hat{v}-\varepsilon}{v+\varepsilon}+\dfrac{(\hat{v}-\varepsilon)^2}{(v+\varepsilon)^2}+\cdots\right\rgroup-n_0&=&n_k\dfrac{(\hat{v}-\varepsilon)(v+\varepsilon)}{v-\hat{v}+2\varepsilon}-n_0\\ &\geq& n_k\left(\dfrac{v \cdot \hat{v}}{v-\hat{v}}-\varepsilon_1\right).
\end{eqnarray*}   
   Among the digits $a_1\cdots a_{m_k}$, there are $k$ blocks of digits which are \textquoteleft free\textquoteright. Denote their lengths by $l_1,\cdots,l_k$. There is an small number $\varepsilon_2$ such that $$\sum_{i=1}^kl_i\leq n_k-n_k\left(\dfrac{v\cdot \hat{v}}{v-\hat{v}}-\varepsilon_1\right)= n_k(1+\varepsilon_2)\dfrac{v-\hat{v}-v\cdot\hat{v}}{v-\hat{v}}.$$ By Theorem \ref{cardinality}, there are at most $\beta\cdot\beta^{l_i}/(\beta-1)$ ways to choose the block with length $l_i$. Thus, one has in total at most $$\left(\dfrac{\beta}{\beta-1}\right)^k\cdot\beta^{\sum_{i=1}^kl_i}\leq \left(\dfrac{\beta}{\beta-1}\right)^k\cdot\beta^{n_k(1+\varepsilon_2)(v-\hat{v}-v\cdot\hat{v})/(v-\hat{v})}$$ possible choices of the digits $a_1\cdots a_{m_k}$. On the other hand, there are at most $k (k\leq C\log_\beta n_k)$ blocks of $0$ in the prefix of length $n_k$ of the infinite sequence $a_1a_2\cdots$. Since there are at most $n_k$ possible choices for their first index, one has in total at most $(n_k)^{C\log_\beta n_k}$ possible choices. Consequently, the set $$\left\{x\in[0,1]:\mathcal{V}_\beta(x,x_0)= v\right\}\cap\left\{x\in[0,1]:\hat{\mathcal{V}}_\beta(x,x_0)=\hat{v}\right\}$$ is covered by $$\left(\dfrac{\beta n_k}{\beta-1}\right)^{C\log_\beta n_k}\cdot\beta^{n_k(1+\varepsilon_2)(v-\hat{v}-v\cdot\hat{v})/(v-\hat{v})}$$ basic intervals of length at most $\beta^{-m_k}$. Moreover, by (\ref{E4}), there is a small number $\varepsilon_3>0$ such that $$\beta^{-m_k}\leq \beta^{-(1+v)(1-\varepsilon_3)n_k}.$$ Take $\varepsilon'=\max\{\varepsilon_2,~\varepsilon_3\}$, we consider the series $$\sum_{N\geq1}(N)^{C\log_\beta N}\beta^{N(1+\varepsilon')(v-\hat{v}-v\cdot \hat{v})/(v-\hat{v})}\beta^{-(1+v)(1-\varepsilon')Ns}.$$ The critical exponent $s_0$ such that the series converges if $s>s_0$ and diverges if $s<s_0$ is given by $$s_0=\dfrac{1+\varepsilon'}{1-\varepsilon'}\cdot\dfrac{v-\hat{v}-v\cdot\hat{v}}{(1+v)(v-\hat{v})}.$$ 
   
   By a standard covering argument and the arbitrariness of $\varepsilon'$, the Hausdorff dimension of the set $$\left\{x\in[0,1]:\mathcal{V}_\beta(x,x_0)= v\right\}\cap\left\{x\in[0,1]:\hat{\mathcal{V}}_\beta(x,x_0)=\hat{v}\right\}$$ is at most equal to $$\dfrac{v-\hat{v}-v\cdot\hat{v}}{(1+v)(v-\hat{v})}.$$
   
\paragraph{{\bf Lower bound}} 
   To obtain the lower bound, we will construct a suitable Cantor type set. For $v\in(0,\infty)$ and $\hat{v}\in(0,1)$ with $v\geq\hat{v}/(1-\hat{v})$, let $$n'_k=\left\lfloor\left(\dfrac{v}{\hat{v}}\right)^k\right\rfloor,\quad m'_k=\lfloor(1+v)n'_k\rfloor,\quad k=1,2,\cdots.$$ Making an adjustment, we can choose two subsequences $\{n_k\}$ and $\{m_k\}$ with $n_k<m_k<n_{k+1}$ for every $k\geq1$ such that $\{m_k-n_k\}$ is a non-decreasing sequence and 
\begin{equation}\label{lowle}
   \lim_{k\rightarrow\infty}\dfrac{m_k-n_k}{n_k}=v,\qquad \lim_{k\rightarrow\infty}\dfrac{m_k-n_k}{n_{k+1}}=\hat{v}.
\end{equation}   
   Recall that the $\beta$-expansion of the fixed point $x_0$ is $$d_\beta(x_0)=(\epsilon_1,\epsilon_2,\cdots,\epsilon_n,\cdots).$$ Consider the set of real numbers $x\in[0,1]$ whose $\beta$-expansion $$x=\dfrac{a_1}{\beta}+\dfrac{a_2}{\beta^2}+\cdots+\dfrac{a_n}{\beta^n}+\cdots,$$ satisfies that for all $k\geq1$, $$a_{n_k}>1,~a_{n_k+1}=\epsilon_1,~a_{n_i+2}=\epsilon_2,~\cdots,~a_{m_i-1}=\epsilon_{m_i-n_i-1},~a_{m_k}>0,$$ $$a_{m_k+(m_k-n_k)}=a_{m_k+2(m_k-n_k)}=\cdots=a_{m_k+t_k(m_k-n_k)}=1,$$ where $t_k$ is the largest integer such that $m_k+t_k(m_k-n_k)<n_{k+1}$. Then, $$t_k\leq\dfrac{n_{k+1}-m_k}{m_k-n_k}\leq\dfrac{2}{\hat{v}},$$ for $k$ large enough. Therefore, the sequence $\{t_k:k\geq1\}$ is bounded. Fix $N$, let $\beta_N$ be the real number defined by the infinite $\beta$-expansion of $1$ as equality (\ref{ED1}). We replace the digit $1$ for $a_{n_k},~a_{m_k}$ and $a_{m_k+i(m_k-n_k)}$ for any $1\leq i\leq t_k$ by the block $0^N10^N$. Fill other places by blocks belonging to $\Sigma_{\beta_N}$. Thus, we have constructed the Cantor type subset $E$. Since $\{t_k\}$ is bound, one has $$\lim_{k\rightarrow\infty}\dfrac{m_k-n_k-1+2N}{n_k+(4k-2)N+\sum_{i=1}^{k-1}2Nt_i}=\lim_{k\rightarrow\infty}\dfrac{m_k-n_k}{n_k}=v,$$ $$\lim_{k\rightarrow\infty}\dfrac{m_k-n_k-1+2N}{n_{k+1}+(4k+2)N+\sum_{i=1}^k2Nt_i}=\lim_{k\rightarrow\infty}\dfrac{m_k-n_k}{n_{k+1}}=\hat{v}.$$ According to the construction, the sequence $d_\beta(x)$ is in $\Sigma_{\beta_N}.$
   
   We distribute the mass uniformly when meet a block in $\Sigma_{\beta_N}$ and keep the mass when go through the positions where the digits are determined by construction of $E$. The Bernoulli measure $\mu$ on $E$ is defined as follows.
   
   If $n<n_1$, then define $\mu(I_n)=1/\sharp\Sigma^n_{\beta_N}$. If $n_1\leq n\leq m_1+4N$, then define $\mu(I_n)=1/\sharp\Sigma^{n_1-1}_{\beta_N}$. If there is an integer $t$ with $0\leq t\leq t_1-1$ such that $$m_1+4N+(t+1)(m_1-n_1)+2Nt<n\leq m_1+4N+(t+1)(m_1-n_1)+2N(t+1),$$ then define $$\mu(I_n)=\dfrac{1}{\sharp\Sigma^{n_1-1}_{\beta_N}}\cdot\dfrac{1}{\left(\sharp\Sigma^{m_1-n_1-1}_{\beta_N}\right)^{t+1}}.$$ If there is an integer $t$ with $0\leq t\leq t_1$ such that $$m_1+4N+t(m_1-n_1)+2Nt<n\leq c,$$ where $c:=\min\{n_2+4N+2Nt_1, m_1+4N+(t+1)(m_1-n_1)+2Nt\}$, then define $$\mu(I_n)=\dfrac{1}{\sharp\Sigma^{n_1-1}_{\beta_N}}\cdot\dfrac{1}{\left(\sharp\Sigma^{m_1-n_1-1}_{\beta_N}\right)^t}\cdot\dfrac{1}{\sharp\Sigma^{n-(m_1+4N+t(m_1-n_1)+2Nt)}_{\beta_N}}.$$
   
   For $k\geq2$, let $$l_k:=n_k+4(k-1)N+\sum^{k-1}_{i=1}2Nt_i,\quad h_k:=m_k+4kN+\sum^{k-1}_{i=1}2Nt_i,$$ $$p_k:=m_k-n_k-1,\quad q_k:=h_k+t_k(m_k-n_k)+2Nt_k.$$ If $l_k\leq n\leq h_k$, then define $$\mu(I_n)=\dfrac{1}{\sharp\Sigma^{n_1-1}_{\beta_N}}\cdot\dfrac{1}{\prod^{k-1}_{i=1}\left(\sharp \Sigma^{p_i}_{\beta_N}\right)^{t_i}\cdot\left(\sharp\Sigma^{l_{i+1}-q_i-1}_{\beta_N}\right)}=\mu(I_{l_k})=\mu(I_{h_k}).$$ If there is an integer $t$ with $0\leq t\leq t_k-1$ such that $$h_k+(t+1)(m_k-n_k)+2Nt<n\leq h_k+(t+1)(m_k-n_k)+2N(t+1),$$ then define $$\mu(I_n)=\mu(I_{h_k})\cdot\dfrac{1}{\left(\sharp\Sigma^{p_k}_{\beta_N}\right)^{t+1}}.$$ If there is an integer $t$ with $0\leq t\leq t_k$ such that $$h_k+t(m_k-n_k)+2Nt<n\leq\min\{l_{k+1}, h_k+(t+1)(m_k-n_k)+2Nt\},$$ then define $$\mu(I_n)=\mu(I_{h_k})\cdot\dfrac{1}{\left(\sharp\Sigma^{p_k}_{\beta_N}\right)^t}\cdot\dfrac{1}{\sharp\Sigma^{n-(h_k+t(m_k-n_k)+2Nt)}_{\beta_N}}.$$
   
   By the construction and Proposition \ref{full}, $I_{h_k}$ is full. For calculating the local dimension of $\mu$, we discuss different cases as follows.
   
   {\bf Case $A$:} If $n=h_k$, then
\begin{eqnarray*}
   \liminf_{k\rightarrow\infty}\dfrac{\log_\beta\mu(I_{h_k})}{\log_\beta \lvert I_{h_k}\rvert}&=&\liminf_{k\rightarrow\infty}\dfrac{n_1-1+\sum\limits_{i=1}^{k-1}\left(t_ip_i+l_{i+1}-q_i-1\right)}{h_k}\cdot\log_\beta\beta_N\\&=&\liminf_{k\rightarrow\infty}\dfrac{n_1-1+\sum\limits_{i=1}^{k-1}\left(l_{i+1}-h_i-2Nt_i-1\right)}{h_k}\cdot\log_\beta\beta_N.
\end{eqnarray*} 
   Recall that $\{t_k:k\geq1\}$ is bounded and $\{m_k:k\geq1\}$ grows exponentially fast in terms of $k$, therefore, $$\liminf_{k\rightarrow\infty}\dfrac{\log_\beta\mu(I_{h_k})}{\log_\beta \lvert I_{h_k}\rvert}=\liminf_{k\rightarrow\infty}\dfrac{\sum_{i=1}^{k-1}\left(n_{i+1}-m_i\right)}{m_k}\log_\beta\beta_N.$$ By equalities (\ref{lowle}), one has $$\lim_{k\rightarrow\infty}\dfrac{m_k}{n_k}=1+v,\quad \lim_{k\rightarrow\infty}\dfrac{m_{k+1}}{m_k}=\dfrac{v}{\hat{v}},\quad\lim_{k\rightarrow\infty}\dfrac{n_{k+1}}{m_k}=\dfrac{v}{\hat{v}(1+v)}.$$ According to Stolz-Ces\`{a}ro Theorem,
\begin{eqnarray*}
   	\liminf_{k\rightarrow\infty}\dfrac{\sum_{i=1}^{k-1}\left(n_{i+1}-m_i\right)}{m_k}&=&\liminf_{k\rightarrow\infty}\dfrac{n_{k+1}-m_k}{m_{k+1}-m_k}\\&=&\liminf_{k\rightarrow\infty}\dfrac{\dfrac{n_{k+1}}{m_k}-1}{\dfrac{m_{k+1}}{m_k}-1}=\dfrac{v-\hat{v}-v\cdot\hat{v}}{(1+v)(v-\hat{v})}.
\end{eqnarray*} 
   Thus, $$\liminf_{k\rightarrow\infty}\dfrac{\log_\beta\mu(I_{h_k})}{\log_\beta\lvert I_{h_k}\rvert}=\dfrac{v-\hat{v}-v\cdot\hat{v}}{(1+v)(v-\hat{v})}\cdot\log_\beta\beta_N.$$
   
   {\bf Case $B$:} For an integer $n$ large enough, if there is $k\geq2$ such that $l_k\leq n\leq h_k$, then $$\liminf_{k\rightarrow\infty}\dfrac{\log_\beta\mu(I_n)}{\log_\beta\lvert I_n\rvert}\geq \liminf_{k\rightarrow\infty}\dfrac{\log_\beta\mu(I_n)}{\log_\beta\lvert I_{h_k}\rvert}=\liminf_{k\rightarrow\infty}\dfrac{\log_\beta\mu(I_{h_k})}{\log_\beta\lvert I_{h_k}\rvert}.$$
   
   {\bf Case $C$:} For an integer $n$ large enough, if there is an integer $t$ with $0\leq t\leq t_k-1$ such that $$h_k+(t+1)(m_k-n_k)+2Nt<n\leq h_k+(t+1)(m_k-n_k)+2N(t+1),$$ then one has $$\mu(I_n)\leq\mu(I_{h_k})\cdot\beta^{-(t+1)p_k}_N.$$ Since $I_{h_k}$ is full, by Proposition \ref{fullc}, $\lvert I_n\rvert=\lvert I_{h_k}\rvert\cdot\lvert I_{n-h_k}(\omega')\rvert$, where $\omega'$ is an admissible block in $\Sigma^{n-h_k}_{\beta_N}$. By Lemma \ref{length}, $$\lvert I_n\rvert\geq\lvert I_{h_k}\rvert\cdot\beta^{-(n-h_k+N)}.$$ Hence, $$\dfrac{-\log_\beta\mu(I_n)}{-\log_\beta\lvert I_n\rvert}\geq\dfrac{-\log_\beta\mu(I_{h_k})+(t+1)p_k\log_\beta\beta_N}{-\log_\beta\lvert I_{h_k}\rvert+((t+1)p_k+N(2t+1))}\geq\dfrac{-\log_\beta\mu(I_{h_k})}{-\log_\beta\lvert I_{h_k}\rvert}\cdot\varphi(N),$$ where $\varphi(N)<1$ and $\varphi(N)$ tends to $1$ as $N$ tends to infinity. If there is an integer $t$ with $0\leq t\leq t_k$ such that $$h_k+t(m_k-n_k)+2Nt<n\leq\min\{l_{k+1}, h_k+(t+1)(m_k-n_k)+2Nt\},$$ then letting $l:=n-(h_k+t(m_k-n_k)+2Nt)$, one has $$\mu(I_n)\leq\mu(I_{h_k})\cdot\beta^{-tp_k-l}_N.$$ Since $I_{h_k}$ is full, by Proposition \ref{fullc}, $\lvert I_n\rvert=\lvert I_{h_k}\rvert\cdot\lvert I_{n-h_k}(\omega')\rvert$, where $\omega'$ is an admissible block in $\Sigma^{n-h_k}_{\beta_N}$. By Lemma \ref{length}, $\lvert I_{n-h_k}(\omega')\rvert\geq\beta^{-(n-h_k+N)}.$ Therefore, $$\lvert I_n\rvert\geq\lvert I_{h_k}\rvert\cdot\beta^{-(n-h_k+N)}.$$ Hence, $$\dfrac{-\log_\beta\mu(I_n)}{-\log_\beta\lvert I_n\rvert}\geq\dfrac{-\log_\beta\mu(I_{h_k})+(tp_k+l)\log_\beta\beta_N}{-\log_\beta\lvert I_{h_k}\rvert+(tp_k+l+t+N(2t+1))}\geq\dfrac{-\log_\beta\mu(I_{h_k})}{-\log_\beta\lvert I_{h_k}\rvert}\cdot\varphi(N).$$
   
   Therefore, in all cases, $$\liminf_{k\rightarrow\infty}\dfrac{\log_\beta\mu(I_n)}{\log_\beta\lvert I_n\rvert}\geq\dfrac{v-\hat{v}-v\cdot\hat{v}}{(1+v)(v-\hat{v})}\cdot\log_\beta\beta_N\cdot\varphi(N).$$ Given a point $x\in E$, let $r$ be a number with $\lvert I_{n+1}(x)\rvert\leq r<\lvert I_n(x)\rvert$. We consider the ball $B(x,r)$. By Lemma \ref{length}, every $n$-th order basic interval $I_n$ satisfies $\lvert I_n\rvert\geq \beta^{-(n+N)}$. Hence, the ball $B(x,r)$ interests at most $\lfloor 2\beta^N\rfloor+2$ basic intervals of order $n$. On the other hand, $$r\geq\lvert I_{n+1}(x)\rvert\geq\beta^{-(n+1+N)}=\beta^{-(1+N)}\cdot\beta^{-n}\geq\beta^{-(1+N)}\cdot\lvert I_n(x)\rvert.$$ Therefore, $$\liminf_{r\rightarrow0}\dfrac{\log_\beta\mu(B(x,r))}{\log_\beta r}=\liminf_{n\rightarrow\infty}\dfrac{\log_\beta\mu(I_n(x))}{\log_\beta\lvert I_n(x)\rvert}.$$ Let $N$ tend to infinity, by Mass Distribution Principle \cite[p.p. 60]{FK90}, we get the lower bound $$\dfrac{v-\hat{v}-v\cdot\hat{v}}{(1+v)(v-\hat{v})}.$$
   
   Hence, the proof of Theorem \ref{A} is complete. Now, we prove Theorem \ref{B}. 
\begin{proof}[Proof of Theorem \ref{B}]
   If $\hat{\mathcal{V}}_\beta(x,x_0)=0$, by Lemma \ref{spezero}, the set $$\left\{x\in[0,1]:\hat{\mathcal{V}}_\beta(x,x_0)=0\right\}$$ is of full Lebesgue measure. If $\hat{\mathcal{V}}_\beta(x,x_0)>1$, by Lemma \ref{infinite} and Lemma \ref{empty}, the set $$\left\{x\in[0,1]:\hat{\mathcal{V}}_\beta(x,x_0)>1\right\}=\left\{x\in[0,1]:\hat{\mathcal{V}}_\beta(x,x_0)=\infty\right\}$$ is countable. 
   
   If $\hat{v}\in(0,1)$, for any $v\geq\hat{v}/(1-\hat{v})$ and any positive integer $L$ large enough, by the similar discussion with upper bound in the proof of Theorem \ref{A}, the Hausdorff dimension of the set $$\left\{x\in[0,1]:v\leq\mathcal{V}_\beta(x,x_0)< v+1/L\right\}\cap\left\{x\in[0,1]:\hat{\mathcal{V}}_\beta(x,x_0)=\hat{v}\right\}$$ is at most equal to $$\dfrac{v-\hat{v}-v\cdot\hat{v}}{(1+v)(v-\hat{v})}+\dfrac{\hat{v}^2}{L(1-\hat{v})}.$$ Let $L$ tend to $\infty$, regard the formula as a function of $v$ with $v\geq\hat{v}/(1-\hat{v})$, the maximum is attained for $v=2\hat{v}/(1-\hat{v})$. Thus, $${\rm dim}_H\left( \left\{x\in[0,1]:\hat{\mathcal{V}}_\beta(x,x_0)=\hat{v}\right\}\right)\leq\left(\dfrac{1-\hat{v}}{1+\hat{v}}\right)^2.$$ On the other hand, from the similar discussion with lower bound in the proof of Theorem \ref{A}, we also have $${\rm dim}_H\left( \left\{x\in[0,1]:\hat{\mathcal{V}}_\beta(x,x_0)=\hat{v}\right\}\right)\geq\left(\dfrac{1-\hat{v}}{1+\hat{v}}\right)^2.$$ Thus, $${\rm dim}_H\left( \left\{x\in[0,1]:\hat{\mathcal{V}}_\beta(x,x_0)=\hat{v}\right\}\right)=\left(\dfrac{1-\hat{v}}{1+\hat{v}}\right)^2.$$ 
\end{proof}

\section{Proofs of Theorems \ref{C} and \ref{D}}
   
   Following the approach of Persson and Schmeling \cite{PS08}, we take a correspondence between the $\beta$-shift and the parameter. Parry \cite[Lemma 2]{P1960} characterized the $\beta$-expansion of $1$.
\begin{Theorem}\label{adone}
	A sequence $(\omega_1,\omega_2,\cdots,\omega_n,\cdots)$ is the $\beta$-expansion of $1$ for some $\beta>1$ if and only if it is self-admissible.
\end{Theorem}   
    
\paragraph{{\bf Upper bound}} 
   We consider a interval $(\beta_0,\beta_1)$, where $1<\beta_0<\beta_1$. For $v\in(0,\infty)$ and $\hat{v}\in(0,1)$, let $$\mathcal{L}_{v,\hat{v}}:=\left\{\beta>1:\mathcal{V}_\beta(1,x_0)=v\right\}\cap\left\{\beta>1:\hat{\mathcal{V}}_\beta(1,x_0)=\hat{v}\right\},$$ $$\mathcal{L}_{v,\hat{v}}(\beta_0,\beta_1):=\left\{\beta\in(\beta_0,\beta_1):\mathcal{V}_\beta(1,x_0)=v\right\}\cap\left\{\beta\in(\beta_0,\beta_1):\hat{\mathcal{V}}_\beta(1,x_0)=\hat{v}\right\}.$$ By Theorem \ref{adone}, each self-admissible sequence corresponds to a real number $\beta>1$. Assume that $\mathcal{S}_{\beta_1}$ is the set of all self-admissible sequences in $\Sigma_{\beta_1}$ and $\pi_{\beta_1}$ is the natural projection from the $\beta$-shift to interval $[0,1]$. Thus, there exists a one-to-one map $$\rho_{\beta_1}:\pi_{\beta_1}(\mathcal{S}_{\beta_1})\rightarrow(1,\beta_1).$$
   
   Define the subset $\mathcal{D}_{v,\hat{v}}$ of $\Sigma_{\beta_1}$ as $$\pi^{-1}_{\beta_1}\left(\left\{x\in[0,1]:\mathcal{V}_\beta(x,x_0)= v\right\}\cap\left\{x\in[0,1]:\hat{\mathcal{V}}_\beta(x,x_0)=\hat{v}\right\}\right).$$ The H\"{o}lder exponent of the restriction of the map $\rho_{\beta_1}$ on $\pi_{\beta_1}(\mathcal{S}_{\beta_1}\cap\mathcal{D}_{v,\hat{v}})$ equals $\log\beta_0/\log\beta_1$. Since $\mathcal{L}_{v,\hat{v}}(\beta_0,\beta_1)\subseteq\rho_{\beta_1}(\pi_{\beta_1}(\mathcal{S}_{\beta_1}\cap\mathcal{D}_{v,\hat{v}}))$, $${\rm dim}_H\mathcal{L}_{v,\hat{v}}(\beta_0,\beta_1)\leq{\rm dim}_H\rho_{\beta_1}(\pi_{\beta_1}(\mathcal{S}_{\beta_1}\cap\mathcal{D}_{v,\hat{v}}))\leq\dfrac{\log\beta_1}{\log\beta_0}{\rm dim}_H\pi_{\beta_1}(\mathcal{S}_{\beta_1}\cap\mathcal{D}_{v,\hat{v}}).$$ By Theorem \ref{A}, letting $\beta_1$ tend to $\beta_0$, if $v<\hat{v}/(1-\hat{v})$, the set $\mathcal{L}_{v,\hat{v}}(\beta_0,\beta_1)$ is empty. Otherwise, $${\rm dim}_H\mathcal{L}_{v,\hat{v}}(\beta_0,\beta_1)\leq\dfrac{v-\hat{v}-v\hat{v}}{(1+v)(v-\hat{v})}.$$     
  
\paragraph{{\bf Lower bound}} 
   Take $\beta_2$ with $1<\beta_0<\beta_1<\beta_2$ such that the $\beta_2$-expansion of $1$ ends with zeros. Thus, the $\beta$-shift $\sum_{\beta_2}$ is a subshift of finite type. Bugeaud and Liao gave a way to calculate the lower bound of the Hausdorff dimension of $\mathcal{L}_{v,\hat{v}}(\beta_0,\beta_1)$.
\begin{Theorem}(\cite[Theorem 5.1]{YLiao2016})\label{lowb}
   Given real numbers $1<\beta_0<\beta_1<\beta_2$. For any $v\in(0,\infty)$ and any $\hat{v}\in(0,1)$ with $v\geq\hat{v}/(1-\hat{v})$, one has $${\rm dim}_H\rho^{-1}_{\beta_2}\mathcal{L}_{v,\hat{v}}(\beta_0,\beta_1)\geq\dfrac{v-\hat{v}-v\cdot\hat{v}}{(1+v)(v-\hat{v})}\cdot\dfrac{\log\beta_1}{\log\beta_2}.$$	
\end{Theorem}

   From Theorem \ref{lowb} and Persson and Schmeling \cite[Theorem 14]{PS08}, we have $${\rm dim}_H\mathcal{L}_{v,\hat{v}}(\beta_0,\beta_1)\geq\dfrac{v-\hat{v}-v\cdot\hat{v}}{(1+v)(v-\hat{v})}\cdot\dfrac{\log\beta_1}{\log\beta_2}.$$ Letting $\beta_2$ tend to $\beta_1$, we obtain the lower bound. Thus, we complete the proof of Theorem \ref{C}.
   
   For the proof of Theorem \ref{D}, one can follow from the proof of Theorem \ref{B}. We omit the details.


\end{document}